\documentclass[A4 paper,20 pt]{amsart}
\usepackage[all,poly]{xy}
\usepackage{amsmath,amsthm,amsfonts,latexsym,amssymb, mathrsfs,stmaryrd,graphics,pst-all,tkz-euclide,yfonts,mathtools,tikz,graphicx,wrapfig}
\usepackage[utf8]{inputenc}
\usepackage{tikz-cd}
\usetikzlibrary{positioning}
\usepackage{fullpage}
\usepackage{color}
\usepackage[pagebackref,colorlinks]{hyperref} 
\usepackage{enumerate}
\scrollmode
\theoremstyle{plain}

\newtheorem{theorem}{Theorem}[section]
\newtheorem{lemma}[theorem]{Lemma}
\newtheorem{corollary}[theorem]{Corollary}
\newtheorem{proposition}[theorem]{Proposition}

\theoremstyle{definition}
\newtheorem{definition}[theorem]{Definition}

\newtheorem{remark}[theorem]{Remark}
\newtheorem{remarks and examples}[theorem]{Remarks and Examples}

\title{ Ulrich bundle on blow-up and blow-down}
\author{Anindya Mukherjee, Pabitra Barik}
\date{October 2026}

\begin{document}

\maketitle

\begin{abstract}
Let $X$ be a smooth projective variety  carrying an Ulrich bundle. In the first part of this note we prove the existence of an Ulrich bundle on the  blow-up of $X$ along a subvariety $Z$, under the assumptions that both  $X$ and $Z$ are abelian variety and that $X$ carries an Ulrich bundle. Also we relate the Ulrich wildness and Ulrich complexity of the base  variety to those of the blow-up. Further we investigate some particular cases for blow-ups in which the abelian hypothesis can be avoided while still obtaining Ulrich bundles on the blow-up. In the second part, we address the converse question. In particular, for an abelian variety $X$ and an abelian subvariety $Z$, we show that the existence of an Ulrich bundle on the blow- up implies the existence of an Ulrich bundle on the base under suitable hypothesis.
\end{abstract}

\section{Introduction}
\label{sec1}
Let $X \subseteq \mathbb{P}^n$ be a projective variety endowed with a very ample line bundle $O_X(1)=O_{\mathbb{P}^{n}}(1)\otimes O_X$. All the varieties we are considering over the field of complex numbers $\mathbb{C}$. Also for a smooth projective variety we fix the polarization $H$ to be $O_{X}(1)$.

A coherent sheaf $E$ on $X$ is said to be Ulrich with respect to $O_X(1)$ if $h^{i}(X,E(-i))=h^{i}(X,E(-j+1))=0$ for all $i>0$ and $j<DimX$. If the coherent sheaf is locally free of finite rank, then we call it an Ulrich bundle.\\
   
   The study of Ulrich bundles originated from the theory of  Ulrich modules in commutataive algebra around $80's$ \cite{Goto}. Later, the notion was introduced into algebraic geometry as Ulrich bundles through the paper of Eisenbud and Schreyer\cite{Eisenbud}. Ulrich sheaves also have a nice geometric interpretation in terms of determinantal representations  of hypersurfaces. Motivated by this application, Eisenbud and Schreyer posed  the following fundamental question about Ulrich sheaves in their paper \cite{Eisenbud}.\\
    \textbf{Question}: Does every  projective variety carry an Ulrich sheaf? If so, what is the minimum possible rank of it??\\ 
     In the case of curves, it is known that every projective curve(even singular) carries an Ulrich bundle.\\
    
    However, a very recent breakthrough provided  smooth counterexamples  to this conjecture:  a class of smooth surfaces
    was constructed that admits a  very ample line bundle with respect to which no Ulrich bundles exist \cite{Anghel}. This result has been subsequently generalized to  arbitrary dimension in \cite{Lopez}. Thus, the existence problem for Ulrich  bundles on smooth projective varieties  becomes more subtle, and the choice of  polarization deserves particular attention.\\     
    
       In \cite{Casnati},Casnati and Kim studied  the behaviour of Ulrich bundles under blow ups and the corresponding converse problem. In their first theorem they show that if a smooth variety carries an Ulrich bundle, then its blow up  at a point also carries an Ulrich bundle. In their second theorem they show that if such a blow up carries an Ulrich bundle, then, under suitable hypothesis on the exceptional divisor, the base variety also carries an Ulrich bundle. More precisely, they show that the following  functor from category of Ulrich bundles on blow up of a variety to the  category of Ulrich bundles on the base variety, has a left inverse\\
       \begin{equation*}
           F \rightarrow F(E)
       \end{equation*}

       In this article we generalize this result to the blow up of an abelian variety of dimension $n$, along an abelian subvariety of  codimension $e$. In the first theorem we show that if an abelian variety carries an Ulrich bundle then with suitable hypothesis on the abelian subvariety we can achieve the existence of Ulrich bundle on the blow up. Conversely, we show that if  such a blow-up carries an Ulrich bundle, then, under suitable hypothesis on the exceptional divisor,the base abelian variety also carries an Ulrich bundle. Summarizing we show that the following  functor below between the category of Ulrich bundles on blow up to the category of Ulrich bundles on abelian variety has a left inverse,\\
       \begin{equation*}
           F \rightarrow F((n-e+1)E)
       \end{equation*}

       We further show that  under suitable hypothesis on Ulrich bundles on the base variety, the blow-up is Ulrich wild. Also we can relate Ulrich complexity of the base variety to the blow up in our setting. These results provide evidence towards understanding the behaviour of Ulrich bundles under birational transformations. Our main results are the following,\\
       
\begin{theorem}\label{theorem 1.1}
    Let $X$ be an abelian  variety of dimension $n$ and let $Z \subset X$ be an abelian subvariety of codimension $e$. Assume that $(X,L)$ carries an Ulrich bundle $F$ for some very ample line bundle $L$ on $X$. If $(F|_{Z})$ is also Ulrich on $(Z,L|_{Z})$, then $(\pi^{*}F)((e-n-1)E)$ is Ulrich on $(X',L')$, where $X'$ is the blow up of $X$ along $Z$ and $L'=\pi^{*}L-E$.
    \end{theorem}
    \begin{theorem}\label{theorem 1.2}
       Let $X$ be abelian variety of dimension $n$ andlet  $Z \subset X$ be an abelian subvariety of codimension $e$. Assume that $F$ is Ulrich bundle on $(X',H')$. Assume  $F|_{E} \cong p_{1}^{*}O_{\mathbb{P}^{e-1}}(n-e+1)^{r_{1}} \otimes p_{2}^{*}M$, Where $p_{1}:Z \times \mathbb{P}^{e-1} \rightarrow \mathbb{P}^{e-1}$, $p_{2}:Z \times \mathbb{P}^{e-1} \rightarrow Z$ be two projection maps and $M$ is an Ulrich bundle on $(Z,H|_{Z})$. Then $\pi_{*}(F(n-e+1)E)$ is Ulrich bundle on $(X,H)$, where $H'=\pi^{*}H-E$. 
\end{theorem}

 \begin{remark}
     Note that in theorem \ref{theorem 1.1}, If we assume that $(X,H)$ carries an Ulrich bundle then proposition \ref{proposition 2.3}, guarantees the existence of Ulrich bundle on $(X,L)$ where $L=H^{\otimes d}$ for some $d>>0$. Thus in the proof of theorem \ref{theorem 1.1} we may replace $H$ by a sufficiently high tensor power and  assume that $(X,L)$, carries an Ulrich bundle. 
 \end{remark}
 \begin{remark}
     Note that as mentioned above for smooth projective varieties showing existence of Ulrich bundle with respect to suitable polarization needs special attention. In our theorem \ref{theorem 1.1}  and theorem \ref{theorem 1.2}, we do not claim the existence of Ulrich bundle on the respective varieties with respect to any polarization. It is clearly induced from the polarization from the other variety depending on the context.\\     
 \end{remark}
 \textbf{Organization of the paper}: The paper is organized as follows. In section \ref{sec 2}, we recall some preliminary facts about Ulrich bundles. In section \ref{sec 3}, we state and prove a known but essential result concerning blow -up of smooth projective varieties. In section \ref{sec 4}, We prove theorem \ref{theorem 1.1} and a corollary related to that. In section \ref{sec 5},  we prove another result regarding existence of Ulrich bundle on blow-up. In section\ref{sec 6}, we discuss Ulrich complexity for blow-up and relate Ulrich complexity of base with the blown-up variety. In section \ref{sec 7}, we discuss Ulrich wildness on the blow-up. In section \ref{sec 8}, we prove theorem \ref{theorem 1.2}.\\

\section{Preliminary results}
\label{sec 2}
\begin{definition}\label{definition 2.1}
 Let $X$ be a projective variety and let  $O_X(1)$ be a very ample line bundle on $X$. A coherent sheaf $F$ on $(X,O_X(1))$ is said to be Ulrich if $H^{i}(X,F(-j))=0$ for all $i\geq 0$ and $1\leq j \leq Dim X$. If the coherent sheaf is locally free of finite rank, then we call it an Ulrich bundle.
 \end{definition}
 \begin{proposition}\label{proposition 2.2}
 Let $X$ and $Y$ be two projective varieties. Let $E$ and $F$ are two Ulrich bundles on $(X,O_X(1))$ and $(Y,O_Y(1))$ respectively, then $E \boxtimes F(n)$ is Ulrich for $(X \times Y,O_X(1) \boxtimes O_Y(1))$ where $n:=Dim X$.
 
 \end{proposition}
 \begin{proof}
     see \cite{Arnuad}
 \end{proof}
 
\begin{proposition}\label{proposition 2.3}
     Let $X \subset \mathbb{P}^{m}$ be a variety of dimension $n$, carrying an Ulrich bundle of rank $r$ for $(X,O_X(1))$. Then $(X, O_{X}(d))$ also carries an Ulrich bundle of rank $rn!$ for any $d\geq 1$. 
 \end{proposition}
 \begin{proof}
 see \cite{Arnuad}
 \end{proof}
 \begin{proposition}\label{proposition 2.4}
     Let $X$ be a smooth variety and let $Y \subseteq X$ be a smooth irreducible subvariety with sheaf of ideals $I_{Y}$.Then we have the following short exact sequence,\\
     \begin{equation*}
         0 \rightarrow I_{Y}/I_{Y}^{2} \rightarrow \Omega_{X} \otimes O_{Y}\rightarrow \Omega_{Y} \rightarrow 0
     \end{equation*}
 \end{proposition}
 \begin{proof}
 see \cite[page 178]{Hart}
 \end{proof}

 \begin{proposition}\label{proposition 2.6}
     Let $f:X \rightarrow Y$ be a continuous function between two topological spaces. Let $F$ be a sheaf on $X$ such that $R^{i}f_{*}F=0$ for all $i >0$. Then\\
     \begin{equation*}
         H^{i}(X,F) \cong H^{i}(Y,f_{*}F)
     \end{equation*}
     for all $i\geq0$
     
 \end{proposition}
 \begin{proof}
 see \cite[page 252]{Hart}
 \end{proof}
 
 \section{On blow up of smooth projective variety}\label{sec 3}
 \label{sec 3}
 \begin{definition}
     Let $X$ be a smooth projective variety ,and let $J$ be a a coherent  ideal sheaf on $X$. The blow-up of $X$  along $J$ is a variety $X'$ together with a birational morphism $\pi:X' \rightarrow X$ such that $\pi^{-1}J.O_{X'}$ is an invertible sheaf on $X'$.
 \end{definition}
 Now we state and prove a lemma regarding blow up which is essential in proving \ref{lemma 4.2}. It is well known but we give the proof here for sake of completeness.\\
 \begin{lemma}\label{lemma 3.2}
  Let $X$ be a smooth projective  variety and $\pi:X' \rightarrow X$ be the blow up of $X$ along a smooth subvariety $Z$. Then we have $\pi_{*}O_{X'}(-nE)=I_{Z}^n$ and $R^{i}\pi_{*}O_{X'}(-nE)=0$ for all $i\geq 0$ and $n\geq 1$.
  \end{lemma}
  \begin{proof}
  First we have to show that $\pi_{*}O_{X'}(-nE)=I_{Z}^n$. As both $X$ and $Z$ are smooth by the Zariski main theorem, we have $\pi_{*}O_{X'}=O_{X}$ and $\pi_{*}O_{E}=O_Z$. Now consider the exact sequence,
  \begin{equation*}
      0 \rightarrow O_{X'}(-E) \rightarrow O_{X} \rightarrow O_{E}\rightarrow 0
  \end{equation*}
  applying the functor $\pi_{*}$ and using the surjectivity of the map $O_{X}$ to $O_{Z}$ we conclude the exactness and thus we have the vanishing $\pi_{*}O_{X'}(-E)=I_{Z}$. Now we have to generalize for any $n\geq 1$. Now consider the local description and assume $X=Spec A$ affine. Then $Z$ is cut out by the ideals $I_{Z}=(f_1,f_2,...,f_r)$ and finally we get a surjection $Sym(A^r) \rightarrow \bigoplus_{m\geq 0}I_{Z}^m$. Then we find that if $O_{\mathbb{P}(A^r)}(n)$ is globally generated then we have an isomorphism between $\pi_{*}O_{\mathbb{P}(A^r)}(n)$ and $\pi_{*}O_{X'}(-nE)$. Thus, by identifying $\pi_{*}O_{\mathbb{P}(A^r)}(n)$ and $Sym^{n}(A^r)$ we get that if the global generation condition is satisfied then we get the desired isomorphism. Now the condition of global generation of $O_{\mathbb{P}(A^r)}(n)$  follows from global generation of $O_{\mathbb{P}(A^r)(1)}$. So we get the desired isomorphism.\\
  
  Now we have to show that $R^{i}\pi_{*}O_{X'}(-nE)=0$ for all $i\geq 0$ and $n \geq 1$. Consider the short exact sequence,\\
  \begin{equation*}
      0 \rightarrow O_{X'}(-E) \rightarrow O_{X'}\rightarrow O_{E} \rightarrow 0
  \end{equation*}
  Now to show the vanishing for $n=1$ using this short exact sequence it is enough to show that $R^{i}\pi_{*}O_{E}=0$ for all $i\geq 0$. Now as $X$ and $Z$ are smooth the exceptional divisor is the projectivization of the normal bundle $N_{Z/X}$. As it is locally free so on a trivializing open cover $Spec(A_{i})$ of $Z$ of $X$ we have the morphism $g:\mathbb{P}^{n}_{A_i} \rightarrow Spec A_{i}$. So from the usual vanishing of projective spaces we conclude that $R^{i}g_{*}O(d)=0$ for $i>0$ and $d>-n-1$. This shows that $R^{i}\pi_{*}O_{E}=0$ for all $i>0$. Now using the identity $O_{E}(-nE) \cong O_{E}(n)$ and above short exact sequence we conclude that $R^{i}\pi_{*}O_{X'}(-nE)=0$ for all $n>0$.
  \end{proof}
  \section{Ulrich bundle on blow up for abelian variety and abelian subvariety}\label{sec 4}
  In this section, we give a sufficient condition for the existence of an  Ulrich bundle on blow-up of a projective  variety along a smooth  subvariety, assuming that the base variety itself carries an Ulrich bundle.As a cosequence,  we obtain an existence result for the blow-up of an abelian variety along an abelian subvariety. Before going to proof of the theorem, our first goal is to determine a very ample divisor on the blow-up. We state the following theorem from \cite{BS}.\\

\begin{theorem}\label{theorem 4.1}
Let $L$ be a very ample line bundle on a non-singular projective variety $X$, and let $Y\subseteq X$ be a closed subscheme corresponding to the ideal sheaf $I_Y$. Let $\pi:X' \rightarrow X$ be the blowing-up of $X$ with respect to $I_Y$ and let $E=\pi^{-1}(Y)$ be the exceptional divisor. Assume that $I_Y \otimes L^{\otimes t}$ is globally generated. Then $t'\pi^{*}L-E$ is very ample  for $t'\geq t+1$.
\end{theorem}
Now we state a and prove a lemma which plays the central role in our paper. It gives a necessary and sufficient condition for existence of Ulrich bundle on blow-up variety if the base variety carries an Ulrich bundle.\\

\begin{lemma}\label{lemma 4.2}
Let $X$ be smooth of dimension $n$ and let  $Z \subseteq X$ be a smooth subvariety of codimension $e$. Let $\pi:X' \rightarrow X$ be the blow up of $X$ along $Z$, with exceptional divisor $E$. Choose  $L$ be very ample on $X$ such that, $L'=\pi^{*}L-E$ is also very ample on $X'$. Let $F$ be Ulrich for $(X,L)$. Then,\\

    $(A)$  $(\pi^{*}F)((e-n-1)E)$ is Ulrich for $(X',L')$\\
    
    if and only if\\
    
    (B)  $H^{i}(F(-pL) \otimes I_{Z}^{n+1-e-p})=0$ for $1 \leq p \leq n-e$\(.\)
\end{lemma}
\begin{proof}
First note that,
  $(\pi^{*}F)((e-n-1)E)(-pL')=\pi^{*}(F(-pL) \otimes O_{X'}((p+e-n-1)E)$. We divide the proof in  two cases depending on the range of $p$.\\
  
   Case 1: $1 \leq p \leq n-e$\(.\)\\
  Then we have that $p+e-n-1 <0$. Using the projection formula and lemma \ref{lemma 3.2} we have,
  $R^{j}\pi_{*}(\pi^{*}F(-pL)) \otimes O_{X'}((p+e-n-1)E)=0$ for $j>0$. Now using proposition \ref{proposition 2.6} we get that,
  \begin{equation*}
      H^{i}((\pi^{*}F)((e-n-1)E)(-pL')) \cong H^{i}(\pi^{*}F(-pL)) \otimes O_{X'}((p+e-n-1)E) \cong H^{i}(F(-pL) \otimes I_{Z}^{n+1-e-p})
  \end{equation*}
   Case 2: $n-e+1 \leq p \leq n$\(.\)\\
  
  We have that $0 \leq p+e-n-1 \leq e-1$. Now using \cite[Lemma 1.4]{Ein} we get that,
  \begin{equation*}
     H^{i}((\pi^{*}F)(e-n-1)E)(-pL') \cong H^{i}(\pi^{*}F(-pL) \otimes O_{X'}((p+e-n-1)E) \cong H^{i}(F(-pL))=0
  \end{equation*}
  Therefore, if (A) holds, we get (B) by considering case 1. If (B) holds, we get (A) by considering the cases $1$ and $2$. 
   \end{proof} 
   
   \begin{lemma}\label{lemma 4.3}
  Let $X$ be smooth variety of dimension $n$ with $Z \subseteq X$ be smooth subvariety of codimension $e$. Assume  $\pi:X' \rightarrow X$ be the blow up of $X$ along $Z$ with exceptional diviosr $E$. Choose $t\gg 0$ be such that $I_{Z}(t)$ is globally generated. Fix $L=(t+1)H$, so that $L'=\pi^{*}L-E$ is very ample on $X'$. If $G$ is an Ulrich bundle on $(X,H)$, then proposition \ref{proposition 2.3} yields is an Ulrich bundle  $F$ on $(X,L)$. If $F$ satisfies (B), then there is an Ulrich bundle for $(X',L')$.
  \end{lemma}
  \begin{proof}
      Apply lemma \ref{lemma 4.2}
  \end{proof}
\begin{lemma}\label{lemma 4.4}
 Let $X$ and  $Z$ be smooth abelian varieties as above in  theorem \ref{theorem 1.1}. Suppose $F$ is an Ulrich bundle on $(X,L)$. Then $F|_{Z}$ is Ulrich on $(Z,L|_{Z})$ if and only if
     \begin{equation*}
        H^{i}(X,F(-pL) \otimes I_{Z}^n)=0 
    \end{equation*}
    for $n\geq 1$, $i \geq 0$ and for $1\leq p \leq Dim Z$\(.\)
 \end{lemma}

 \begin{proof}
 We first consider the case $k=0$. Since $I_{Z}^{0}=O_X$,  we have $O_{X}/I_{Z} \cong O_{Z}$. First assume $F|_{Z}$ is an Ulrich bundle on $(Z,L|_{Z})$. So we get,
    \begin{equation*}
        H^{i}(F(-pL) \otimes O_Z)=0 
    \end{equation*}
    for $1 \leq p \leq Dim Z$.
     By proposition \ref{proposition 2.4}, we have that the sheaf $I/I^{2}$ is trivial as both the cotangent bundle of the variety and subvariety are trivial. As $Z$ is a smooth subvariety we have $Sym^{k}(I_Z/I_Z^{2}) \cong I_{Z}^k/I_{Z}^{k+1}$. So $I^{k}/I^{k+1}$ is also trivial for all $k \geq 1$. Assume inductively  that the  desired vanishing holds for $I^{k}$. Let's consider the tensored short exact sequence,
    \begin{equation*}
        0 \rightarrow F(-pL) \otimes I^{k+1} \rightarrow F(-pL) \otimes I^{k} \rightarrow F(-pL) \otimes I^{k}/I^{k+1} \rightarrow 0
    \end{equation*}
   As the sheaf $I^k/I^{k+1}$ is trivial for all $k$ and the restriction is also Ulrich we have that,
    \begin{equation*}
        H^{i}(F(-pL) \otimes I^k/I^{k+1})=0
    \end{equation*}
    So now considering the corresponding long exact sequence of cohomology, we get using inductive hypothesis and above vanishing that,
    \begin{equation*}
        H^{i}(F(-pL) \otimes I^{k+1})=0
    \end{equation*}
    Thus from induction we conclude that,
      \begin{equation*}
        H^{i}(F(-pL) \otimes I^n)=0
    \end{equation*}
    for all $n\geq1$.\\
   
   Conversely, let's assume the vanishing of the cohomology of the ideal sheaf as above. Now consider the tensored short exact sequence,
   
   \begin{equation*}
       0 \rightarrow F(-pL) \otimes I \rightarrow F(-pL) \rightarrow F(-pL) \otimes O_{Z} \rightarrow 0
   \end{equation*}
   
   Now $H^{i}(F(-pL))=0$ for all $i\geq 0$ and $1 \leq p \leq dim X$, as $F$ is Ulrich on $(X,L)$. We have $H^{i}(F(-pL) \otimes I_{Z})=0$ for $i \geq 0$ and $1 \leq p \leq Dim Z$ by hypothesis. So $H^{i}(F(-pL) \otimes O_Z)=0$ for all $i \geq 0$ and $ 1 \leq p \leq Dim Z$. Thus $F|_{Z}$ is Ulrich on $(Z,L|_{Z})$.
  \end{proof}
  \textbf{proof of Theorem 1.1}: Let $F$ be an Ulrich bundle on $(X,L)$, where $L=H^{\otimes d}$ and $d$ is chosen sufficiently large so that $I_{{Z}/X}(dH)$ is globally generated. Now by hypothesis we have  $F|_{Z}$ is Ulrich on $(Z,L|_{Z})$. Then using lemma \ref{lemma 4.4},
  \begin{equation*}
      H^{i}(X, F(-pL) \otimes I_{Z}^{n})=0 \ for \ all \  n\geq 1 \ and \ for \ 1 \leq p \leq DimZ
  \end{equation*}
  So clearly this vanishing condition implies the vanishing condition of lemma \ref{lemma 4.2}. Thus using lemma \ref{lemma 4.3} we can say that $(X',L')$ carries an Ulrich bundle.\\
  
 \begin{corollary}\label{corollary 4.5}
 Let $X=A \times A$ where $A$ is an abelian variety of dimension $n$. Let $Z=A \times \{0\} \subset X$ be an abelian subvariety, with  $X'=Bl_{A \times \{0\}}(A \times A)$ be the blow up. Assume $H$ is a very ample line bundle on $A$. If $(A,H)$ carries an Ulrich bundle, then $X'$  admits an Ulrich bundle with respect to a suitable polarization.
 \end{corollary}
 \begin{proof}
 Let $(A,H)$ carry an Ulrich bundle $E$. Now we have for $t\gg 0$, $I_{{Z}/X}(tH)$ is globally generated. Using proposition \ref{proposition 2.2}  and proposition \ref{proposition 2.3} we can show that  $F \boxtimes F(n(t+1)H)$ to be Ulrich bundle on $(A \times A,(t+1)H \otimes (t+1)H)$. Let's call $L=(t+1)H \otimes (t+1)H$. Since  $A$ is an abelian variety, proposition \ref{proposition 2.4} shows that $I_{Z}/I_{Z}^2$ is trivial. So in view of theorem \ref{theorem 1.1}, it remains to show that the restriction of the bundle $F \boxtimes F(n(t+1)H)$ to $Z=A \times \{0\}$ is an Ulrich with respect to $L|_{Z}$. Consider the embedding ,
 $i:A \rightarrow A \times A$ to be $x: \rightarrow (x,0)$. Now we have\\
 \begin{equation*}
     i^{*}(p^{*}F \otimes q^{*}F(n(t+1)H)) \cong (p \circ i)^{*}F \otimes (q \circ i)^{*}F(n(t+1)H) \cong F \otimes O_{A}^{\oplus r}\cong F^{\oplus r}.
 \end{equation*}
 where $p:A \times A \rightarrow A$ ans $q:A \times A \rightarrow A$ be two projection map. As $F$ is Ulrich on $(A,(t+1)H)$, then $F^{\oplus rkF}$ is also Ulrich on $(A,(t+1)H)$\cite{Arnuad}. So from lemma \ref{lemma 4.2}, it follows that $(\pi^{*}F)((-n-1)E)$ is the required Ulrich bundle on $(X',\pi^{*}L-E)$.
 \end{proof}
 
 \section{Ulrich bundles on blow ups along one dimensional subvarieties}
 \label{sec 5}
 In this section, we give a sufficient condition for the  existence of an Ulrich bundle on the blow up of a smooth projective variety along a smooth one-dimensional subvariety.\\
 
 \begin{theorem}\label{theorem 5.1}
 Let $X$ be a smooth projective variety and let $Z\subset X$ be a  smooth  one dimensional smooth subvariety. $L$ is chosen so that $L'=\pi^{*}L-E$ be as above which is a very ample line bundle on $X'=Bl_{Z}(X)$. Assume  $F$ is an Ulrich bundle on $(X,L)$. Then $(\pi^{*}F)(-2E)$ is Ulrich on $(X',L')$ if and only if $F|_{Z}$ is Ulrich on $(Z,L|_{Z})$.
 \end{theorem}
 \begin{proof}
    First assume that $F|_{Z}$ is Ulrich on $(Z,L|_{Z})$. As  we have Dim$Z=1$, so in view of lemma \ref{lemma 4.2}, we only have to show that $H^{i}(F(-L) \otimes I_Z)=0$ for all $i \geq 0$. Now consider the tensored short exact sequence,
    \begin{equation*}
        0 \rightarrow F(-L) \otimes I_Z \rightarrow F(-L) \rightarrow F(-L) \otimes O_Z \rightarrow 0
    \end{equation*}
    Now as both $F$ and $F|_{Z}$ is Ulrich on $(X,L)$ and $(Z,L|_{Z})$ respectively, we then have that $H^{i}(F(-L) \otimes I_Z)=0$ for all $i \geq 0$.\\
    
   Conversely,
   Let $(\pi^{*}F)(-2E)$ is Ulrich on $(X',L')$. So in view of lemma \ref{lemma 4.4} we have $H^{i}(F(-L) \otimes I_Z)=0$ for all $i \geq 0$. Similarly using the short exact sequence as above we can conclude that $H^{i}(F(-L)|_{Z})=0$ for all $i \geq 0$. Thus $F|_{Z}$ is Ulrich on $(Z,L|_{Z})$. 
\end{proof}
\begin{remark}
Note that although we have a restriction for the codimension of the subvariety in this case, but it does not require any hypothesis on the variety. This may be helpful to construct Ulrich bundle on several blow-ups on the projective variety where existence of Ulrich bundle itself is known.
\end{remark}

\section{Ulrich complexity on blow up}
\label{sec 6}
In this section we discuss about Ulrich complexity of the blow up if the Ulrich complexity is known on the base. We give  a general theorem and then give some examples to support the theorem which is related to our previous observations. But before going into the theorem we first give the definition of Ulrich complexity.\\

\begin{definition}
    Let $X \subset \mathbb{P}^{n}$ be a smooth projective variety. Then Ulrich complexity of $X$, denoted by $UC(X)$, is defined to be the minimum rank of an Ulrich bundle on $X$ with respect to a suitable polarization.

\end{definition}
\begin{remark}
Note that, from the definition of Ulrich complexity it is clear, if  Ulrich bundles on $X$ exist with respect to various polarizations, then we choose the polarization for which $X$ attains  the minimum rank Ulrich bundle.
\end{remark}

\begin{proposition}
    Let $X$ be a smooth projective variety of dimension $n$ with $UC(X)=r$. Suppose that $F$ is an Ulrich bundle on $X$, satisfying the hypothesis of theorem \ref{theorem 1.1} when $X$ is an abelian variety and of theorem \ref{theorem 5.1} in general case. Then $UC(X')\le rn!$
\end{proposition}
\begin{proof}
We show the proof when $X$ is abelian variety and similar technique will follow for the general case. Now let $E$ be an Ulrich bundle of rank $r$ on $(X,H)$, where $H$ is very ample on $X$. Now using proposition \ref{proposition 2.3} we can show that $F$ is an Ulrich bundle on $(X,dH)$ for $d>>0$ of rank $rn!$.  As it satisfies the hypothesis of \ref{theorem 1.1},So we have an Ulrich bundle of rank $rn!$ on $X'$. Thus $UC(X') \le rn!$.
\end{proof}
\begin{corollary}
 Let $A$ be an abelian surface. Then $UC(Bl_{A}(A \times A)\leq 16$.
 \end{corollary}
 \begin{proof}
  we know that, if $A$ is an abelian surface with an very ample line bundle $H$, then by \cite{Arnuad} there exist stable  Ulrich bundle of rank $2$ with respect to $H$. Now using proposition \ref{proposition 2.3}, we can show that there exist an Ulrich bundle of rank $4$ on $(A,L)$. So  from theorem \ref{theorem 1.1}, it is clear that  the rank of the required bundle on $(A \times A,L)$ is $16$. Hence the result follows.
 \end{proof}
 \begin{remark}
 If we assume that $NS(A)$ to be $\mathbb{Z}$, then $A$ admits no Ulrich line bundles. However, this bound can be improved for  particular class of abelian surfaces,  such as $A =C\times C$, where $C$ is an Elliptic curve. As $C$ has Ulrich line bundle with respect to $H$, we have Ulrich line bundle on $A$ with respect to suitable polarization. In that case $UC(Bl_{A}(A \times A)\leq 4$.
 \end{remark}
 \begin{corollary}
Let $C$ be a smooth projective curve. Then $UC(Bl_{C}(C \times C \times C))=1$.    
\end{corollary}
\begin{proof}
     Let  $F$ be an  Ulrich line bundle on $(C,(t+1)H)$ where  $t >>0$ is chosen such that $I_{{C}/{C \times C \times C}}(tH)$ is globally generated.  We have by proposition \ref{proposition 2.2} that  $F \boxtimes F((t+1)H)$ is an Ulrich  line bundle on $(C \times C, (t+1)H \boxtimes (t+1)H)$. Let's call $E=F \boxtimes F((t+1)H)$ and $L=(t+1)H \boxtimes (t+1)H$. Now if we let $X=C \times C$, so consider the two projections,$p_1:C \times X \rightarrow C$ and $p_2:C \times X \rightarrow X$. Now again Using proposition \ref{proposition 2.2} we have $F'=p_1^{*}F \otimes p_2^{*}E((t+1)H)$ is an Ulrich line bundle  for $(C \times X, H \boxtimes L)$. Now fix a point $x_o$ in $C$ and consider the embedding $i: C \rightarrow C \times X$ by $x \rightarrow (x,x_o,x_o)$. Now we have,
    \begin{equation*}
        i^{*}(p_1^{*}F \otimes p_2^{*}E((t+1)H) \cong F
    \end{equation*}
     Thus it is Ulrich on $(C,(t+1)H$. So using theorem \ref{theorem 5.1} we can show the existence of Ulrich line bundle  on $(Bl_{C}(C \times C \times C),\pi^{*}((t+1)H \boxtimes L)-E)$.
\end{proof}
\begin{remark}
    Note that similarly we can consider the variety $Bl_{C}(C^n)$ and show that its Ulrich complexity to be one for any $n\geq 3$.
\end{remark}

\section{Ulrich wildness on blow up}
\label{sec 7}
In this section, we study the behaviour of Ulrich-wildness under blow-ups.  on the blow-up variety if we have certain family Ulrich bundles of large dimension exists on the base variety. Our goal is to transfer sufficiently large  families of Ulrich  bundles from the base variety to the blow-up. First we define Ulrich wildness,\\
\begin{definition}
A smooth projective variety $X \subset \mathbb{P}^{n}$  is said to be Ulrich wild if, with respect to some polarization, it admits a family of pairwise non isomorphic ,indecomposable Ulrich bundles of arbitrary large dimension.
   \end{definition}
Ulrich wildness is a very crucial property of a smooth variety. If a projective variety carries Ulrich bundle , then the most natural questions to be asked about other than Ulrich complexity is Ulrich wildness. It helps us to understand that if a variety carries indecomposable Ulrich bundles, how large the family can be, which is anyway helpful to understand  the moduli of such objects. Now we show that wildness is a birational property in the sense that blow up preservs Ulrich wildness in our setting. Our main tool to show Ulrich wildness is \cite[Theorem 5.1]{Casnati} . 

\begin{proposition}\label{proposition 7.2}
    Let $X$ be a smooth projective variety which carries two simple Ulrich bundles $A$ and $B$ such that,\\
    \begin{align*}
       &h^{0}(A\otimes B^{\vee})=h^{0}(B \otimes A^{\vee})=0\\
       &and\\
       &h^{1}(A \otimes B^{\vee})>2
     \end{align*}  
       
       Assume moreover that s the hypothesis of theorem \ref{theorem 1.1} are satisfied when X is abelian  and those of theorem \ref{theorem 5.1} are satisfied in general case. Then the blow up  $X'$ is Ulrich-wild.  
\end{proposition}
\begin{proof}
Following the hypothesis we know that there are two simple Ulrich bundles $A$ and $B$ which satisfy the necessary conditions. We show the proof when $X$ is abelian variety and the blow up is along abelian subvariety. Also similar technique would follow for the blow-up of any smooth projective variety along one dimensional subvariety as it is independent of the center of blow up.\\
 Now let $F_{i}$ be Ulrich on $(X,L)$ with the assumption that each $F_{i}$ is Ulrich on the abelian subvariety $(Z,L|_{Z})$. Now following the theorem \ref{theorem 1.1} we have $(\pi^{*}F_{i})((e-n-1)E)$ is Ulrich on $X'=Bl_{Z}(X)$ with respect to suitable polarization. So we proceed in steps. We first show that simplicity is preserved under the construction. Suppose that $F$ is simple. Now,\\
\noindent \textbf{preservation of simplicity}\\
\begin{align*}
     End((\pi^{*}F)((n-e+1)E) & \cong \Gamma(X, \mathcal{E}nd((\pi^{*}F)((n-e+1)E)\\
     &\cong \Gamma(X, (\pi^{*}F) \otimes (\pi^{*}F)^{\vee})\\
     &\cong \Gamma(X, \mathcal{E}nd(\pi^{*}F))\\
     &\cong \Gamma(X, \mathcal{E}nd(F)\\
     &\cong End(F)
 \end{align*}
 So if we start with a simple bundle $F$, then $End(F) \cong k$ and from the above we can conclude that $End((\pi^{*}F)((n-e+1)E) \cong k$. Thus it is also simple. So if we start with two non isomorphic simple  bundle $A$,$B$ on $X$, then we can construct non isomorphic simple bundle on $X'$. This satisfies the first criteria for the blow up to be Ulrich wild.\\
 
  \noindent \textbf{Preservation of cohomology}\\
  
  Let $A$ and $B$ be two simple Ulrich which satisfy the hypothesis of the above proposition. From lemma \ref{lemma 3.2} we have,\\
  \begin{equation*}
   R^{i}\pi_{*}O_{X'}=0 \ and \  \pi_{*}O_{X'}=O_{X}.
   \end{equation*}
   So for any locally free sheaf $G$ on $X$, we have that,\\
  \begin{equation*}
      R^{i}\pi_{*}(\pi^{*}G)\cong G \otimes R^{i}\pi_{*}O_{X'}
  \end{equation*}
  As from the property of blow up it is zero so using \cite[page 252]{Hart} we have that,\\
  \begin{equation*}
      H^{i}(X',\pi^{*}G) \cong H^{i}(X,G) \ for \ all \ i \geq 0
  \end{equation*}
  Let's call $A'=(\pi^{*}A)((e-n-1)E)$ and $B'=(\pi^{*}B)((e-n-1)E)$. Thus,\\
  \begin{align*}
      H^{0}(X',A' \otimes B') &\cong H^{0}(\pi^{*}A \otimes (\pi^{*}B)^{\vee})\\
      &\cong H^{0}(X',\pi^{*}(A \otimes B^{\vee}))\\
      &\cong H^{0}(X, A \otimes B^{\vee})
  \end{align*}
  As we have $h^{0}(A \otimes B^{\vee})=0$ so we have $h^{0}(A' \otimes B'^{\vee})=0$. Similarly we have $h^{1}(A' \otimes B'^{\vee})>2$. So we conclude that $X'$ carries two non-isomorphic Ulrich bundle $A'$ and $B'$ which are simple and which satisfy the required cohomology vanishing condition. Thus $X'$ is also Ulrich wild with respect to $H'$.
  \end{proof}
 \begin{corollary}
     Let $X$ be any smooth projective variety which carries simple Ulrich bundle $F_{1}$ and $F_{2}$ with respect to a polarization $H$. If $F_{1}$ and $F_{2}$ satisfy the condition of proposition \ref{proposition 7.2}, then $X'$, the blow up of $X$ at a general point is always Ulrich wild with respect to a suitable polarization.
 \end{corollary}
\begin{proof}
    Let $F_{1}$ and $F_{2}$ be two simple Ulrich bundle on $(X,H)$. In \cite{Secci}, it is shown that $(\pi^{*}F)(-E)$ is Ulrich on $(X',2\pi^{*}H-E)$, if $F$ is Ulrich on $(X,H)$. Using the above technique similarly  we can show that if $F_{1}$ and $F_{2}$ are non-isomorphic  simple Ulrich  bundles on $X$, then $\pi^{*}(F_{1})(-E)$ and $\pi^{*}(F_{2})(-E)$ are also non-isomorphic and simple. Moreover they satisfy the same vanishing cohomology condition. Thus $X'$ is also  Ulrich wild with respect to the polarization $(2\pi^{*}H-E)$.   
\end{proof}
\section{Ulrich bundle on blow down}
\label{sec 8}
In this section, we study the converse problem: under suitable hypothesis,an Ulrich bundle on the blow-up induces an Ulrich bundle on the base variety. We generalize the construction of \cite[page 61]{Costa}to the case of an  abelian variety and  an abelian subvariety under some suitable hypothesis on  the restriction of the  Ulrich bundle to the exceptional divisor. We continue with the notations used in the earlier section. First we determine the exceptional divisor for this blow-up and show that it simplifies significantly as we have  triviality of tangent bundle for both the variety and the subvariety.\\

\begin{proposition}\label{proposition 8.1}
Let $X$ be an abelian variety of dimension $n$ and let $Z \subset X$ be an abelian variety of dimension $r$ with $X'=Bl_{Z}(X)$. Then the exceptional divisor is $\mathbb{P}^{e-1} \times {Z}$, where $e$ is codimension of $Z$ in $X$.
\end{proposition}
\begin{proof}
    Consider the exact sequence,\\
    \begin{equation*}
        0 \rightarrow TZ \rightarrow TX|_{Z} \rightarrow N_{Z/X} \rightarrow 0
    \end{equation*}
    Since both $X$ and $Z$ are abelian varieties, their tangent bundles are  trivial. Thus $N \cong O_{Z}^{e}$. We know that the exceptional divisor $E \cong \mathbb{P}(N^{*})$. Also  
    $\mathbb{P}(N^{*})$ is $Proj Sym N$. For free modules the symmetric algebra is the polynomial ring. Thus we have the exceptional divisor $E \cong Proj O_{Z}[T_1,..T_{e}] \cong \mathbb{P}^{e-1}_{Z}$. We canonically identify this with $\mathbb{P}^{e-1} \times Z$.
\end{proof}
Now we will prove our first lemma. Here we show the vanishing of higher direct image sheaf which is essential to relate the cohomology of any sheaf and its push forward. \\

\begin{lemma}\label{lemma 8.2}
Let $F$ be the Ulrich bundle on the blow-up $(X',H')$ and suppose  it satisfies the hypothesis of theorem \ref{theorem 1.1}. Then  $R^{j}\pi_{*}F((n-e+1)E)=0$ for all $j>0$.
\end{lemma}
\begin{proof}
We have to show that$(R^{j}\pi_{*}F((n-e+1)E)_{x}=0$ for all $x$ in $X$. We divide the proof in two cases. This helps  us to understand the difference of approach to prove each case which is essentially determined by the geometry of blow-up morphism.\\
\textbf{case 1:$x \in X-Z$}\\
We know that, the  blow-up map is an isomorphism outside the subvariety $Z$. So the push forward functor is exact. Thus we can directly conclude that $(R^{j}\pi_{*}F((n-e+1)E)=0$ outside $Z$.\\

\textbf{case 2: $x \in Z$}\\
from proposition \ref{proposition 8.1},   $\pi^{-1}Z \cong Z \times \mathbb{P}^{e-1}$. So the fiber at each point is $\mathbb{P}^{e-1}$.  By assumption we have, $F|_{E} \cong p_{1}^{*}O_{\mathbb{P}^{e-1}}(n-e+1)^{r_{1}} \otimes p_{2}^{*}M$. So now we will compute the restriction of the bundle $F((n-e+1)E)$ on the exceptional divisor $E$.\\
\begin{align*}
    F((n-e+1)E)|_{E} &\cong F|_{E} \otimes O_{E}((n-e+1)E)\\
    &\cong p_{1}^{*}O_{\mathbb{P}^{e-1}}(n-e+1)^{r_{1}} \otimes p_{2}^{*}M \otimes O_{E}(e-n-1)\\
    &\cong p_{1}^{*}O_{\mathbb{P}^{e-1}}^{r_{1}} \otimes p_{2}^{*}M 
\end{align*}
Restricting the blow up morphism to the exceptional divisor we get a proper surjective map $\pi|_{E}:Z \times \mathbb{P}^{e-1} \rightarrow Z$. Now we compute the restriction of this bundle on exceptional divisor.  At each point $z \in Z$, the fiber is $\mathbb{P}^{e-1}$ which has constant dimension.\\
Fix a point $z \in Z$. Consider the  map,\\
\begin{align*}
       i:&\mathbb{P}^{e-1} \rightarrow \mathbb{P}^{e-1} \times Z\\
       & p \rightarrow(p,z)
\end{align*} 
so we compute the pull back of the restriction on the fiber via this morphism.
\begin{align*}
i^{*}(p_{1}^{*}O_{\mathbb{P}^{e-1}}(m+e-n-1)^{r_{1}} \otimes p_{2}^{*}M) & \cong (p_{1} \circ i)^{*}O_{\mathbb{P}^{e-1}}(m+e-n-1)^{r_{1}} \otimes (p_{2} \circ i)^{*}M\\
&\cong O_{\mathbb{P}^{e-1}}(m+e-n-1)^{r_{1}} \otimes O_{\mathbb{P}^{e-1}}^{r_{2}}\\
&\cong O_{\mathbb{P}^{e-1}}^{r}
\end{align*}

So from the  cohomology of structure sheaf on projective space we have,\\
\begin{equation*}
    H^{j}(\mathbb{P}^{e-1},O_{\mathbb{P}^{e-1}})=0 \ for \ all \ j >0
\end{equation*}
Now as for each $z \in Z$, as mentioned above the dimension of fiber is constant and the dimension of the cohomology group on each fiber is constant. So by Gauret's theorem \cite[page 288]{Hart}, we have ,\\
\begin{equation*}
    R^{j}(\pi|_{E})_{*}F \otimes k(z) \cong H^{i}(\mathbb{P}^{e-1},F)
\end{equation*}
for any locally free sheaf $F$. Now taking $F=O_{\mathbb{P}^{e-1}}^{r}$, as the corresponding sheaf cohomology group is zero, so we finally have that,\\
\begin{equation*}
    R^{j}\pi_{*}F((n-e+1)E)=0 \ for \ al \ j>0
\end{equation*}
\end{proof}
    Our next lemma is crucial. It is the main tool  which helps to achieve  the vanishing required to get the above bundle as Ulrich.\\

\begin{lemma}\label{lemma 8.3}
Let $F$ be Ulrich on $(X',H')$ and suppose it satisfies the hypothesis of theorem \ref{theorem 1.1}. Then $H^{i}(F(-pH'-(p-n+e-1)E)=0$ for all $i\geq 0$ and $1 \leq p \leq n$.
\end{lemma}
\begin{proof}
We distinguish two cases according to the range of $p$.\\

\textbf{case 1:\ $n-e+1\leq p \leq n$}\\

Let $F$ be the Ulrich bundle on $(X',H')$ and consider the short exact sequence,\\
\begin{equation*}
    0 \rightarrow F(-aH'-bE) \rightarrow F(-aH'-(b-1)E) \rightarrow F(-aH'-(b-1)E)|_{E} \rightarrow 0
\end{equation*}

Now we compute the restriction of the bundle $F(-aH'-(b-1)E|_{E}$.\\
\begin{align*}
    F(-aH'-(b-1)E)|_{E} & \cong F|_{E} \otimes O_{X'}(-aH')|_{E} \otimes O_{X'}((1-b)E)|_{E}\\
                        & \cong p_{1}^{*}O_{\mathbb{P}^{e-1}}(n-e+1) \otimes p_{2}^{*}M \otimes O_{X'}(-a\pi^{*}H)|_{E} \otimes O_{X'}((a-b+1)E)|_{E}\\
                        &\cong p_{1}^{*}O_{\mathbb{P}^{e-1}}(n-e+1) \otimes p_{2}^{*}M \otimes \pi^{*}O_{X}(-aH)|_{E} \otimes O_{E}(b-a-1)
\end{align*}

Now consider the exceptional divisor $E \cong Z \otimes \mathbb{P}^{e-1}$ with the two projection map $p_{1}:Z \times \mathbb{P}^{e-1} \rightarrow \mathbb{P}^{e-1}$ and $p_{2}:Z \times \mathbb{P}^{e-1} \rightarrow Z$.We have $O_{E}(1) \cong p_{1}^{*}O_{\mathbb{P}^{e-1}(1)}$ and $\pi^{*}O_{X}(-aH)|_{E} \cong p_{2}^{*}O_{Z}(-aH)$. So from the above calculation using this identification we finally have,\\
\begin{align*}
    F(-aH'-(b-1)E)|_{E} \cong p_{1}^{*}O_{\mathbb{P}^{e-1}}(n-e+b-a) \otimes p_{2}^{*}M(-aH)
\end{align*}
Now as taking direct sum commutes with sheaf cohomology it is enough to show  that $H^{i}(Z \times \mathbb{P}^{e-1}, p_{1}^{*}(O_{\mathbb{P}^{e-1}}(n-e+b-a)) \otimes p_{2}^{*}M(-aH))=0$ for all $i \geq 0$, $n-e+1\leq a \leq n$ and all $1 \le b \le a+e-n-1$.
\begin{equation*}
    H^{k}(Z \times \mathbb{P}^{e-1},p_{1}^{*}(O_{\mathbb{P}^{e-1}}(n-e+b-a)) \otimes p_{2}^{*}M(-aH)) \cong \bigoplus_{i+j=k}H^{i}(\mathbb{P}^{e-1},O_{\mathbb{P}^{e-1}}(n-e+b-a)) \otimes H^{j}(Z,M(-aH))
\end{equation*}
 Now mainly we need vanishing of $H^{0}(\mathbb{P}^{i}, O_{\mathbb{P}^{e-1}}(n-e+a-b))=0$ for $i=0$ and $i=e-1$. Now from the inequality $1 \le b \le a+e-n-1$ we have $b-a < e-n$, So the line bundle $O_{\mathbb{P}^{e-1}}(n-e+b-a)$ has negative degree. Thus from usual vanishing of projective space we have,\\
\begin{equation*}
    H^{0}(\mathbb{P}^{e-1}, O_{\mathbb{P}^{e-1}}(n-e+b-a))=0
\end{equation*}
 Now we show that $H^{e-1}(\mathbb{P}^{e-1},O_{\mathbb{P}^{e-1}}(n-e+b-a))=0$.  As $K_{\mathbb{P}^{e-1}} \cong O_{\mathbb{P}^{e-1}}(-e)$, Using Serre duality we have,\\
\begin{equation*}
    H^{e-1}(\mathbb{P}^{e-1}, O_{\mathbb{P}^{e-1}}(n-e+b-a)) \cong H^{0}(\mathbb{P}^{e-1}, O_{\mathbb{P}^{e-1}}(a-b-n))^{\vee}
\end{equation*}
Since, as $a \leq n$ and $b \geq 1$, we have $a-b<n$. So the degree of the above line bundle is negative. Thus we have that,\\
\begin{equation*}
    H^{e-1}(\mathbb{P}^{e-1}, O_{\mathbb{P}^{e-1}}(n-e+b-a))=0
\end{equation*}

 Thus combining all the results we get that, $H^{k}(E,F(-aH'-(b-1)E)|_{E}=0$ for all $k\geq 0$ and  $n-e+1 \leq a \leq n$ and $1 \le b \le a+e-n-1$\\

Now Let's consider the long exact sequence of cohomology by,\\
\begin{align*}
    0 \rightarrow \Gamma(X',F(-aH'-bE)) \rightarrow \Gamma(X',F(-aH'-(b-1)E) \rightarrow \Gamma(E,F(-aH'-(b-1)E)|_{E})&\\ \rightarrow H^{1}(X',F(-aH'-bE)\rightarrow ...
\end{align*}
Now as $F$ is Ulrich on $(X',H')$, So we have that $H^{i}(X',F(-aH'))=0$ for $1\leq a \leq n$. So vanishing holds automatically for $b=0$. So we can start the induction for $b \ge 1$. We see that for $b=1$, we have $H^{i}(F(-aH'-bE)=0$ for $1\leq a\leq n-e$. Let's assume that $H^{i}(F(-aH'-kE))=0$ for some $b=k$. Now using the above long exact sequence of cohomology and the inductive hypothesis we have that $H^{i}(F(-aH'-(k+1)E)=0$ for $b=k+1$. Thus finally we have that,$H^{i}(F(-aH'-bE)=0$ for $n-e+1\leq a\leq n$ and for $1 \le b \le a+e-n-1$. Thus the induction hypothesis works. So combining the common range for $a$ and $b$ we get that,\\ 
\begin{equation*}
    H^{i}(X', F(-pH'-(p-n+e-1)E)=0 \ \forall i \geq 0 \ and \  \ n-e+1\leq p \leq n
\end{equation*}
\textbf{case 2:\ $1\leq p \leq n-e$}\\

Let $F$ be Ulrich on $(X',H')$ and consider the short exact sequence,\\
\begin{equation*}
    0 \rightarrow F(-aH'+(b-1)E) \rightarrow F(-aH'+bE) \rightarrow F(-aH'+bE)|_{E} \rightarrow 0
\end{equation*}
Now similarly as above we compute the restriction of the bundle $F(-aH'+bE)|_{E}$.\\
\begin{align*}
    F(-aH'+bE)|_{E} &\cong F|_{E} \otimes O_{X'}(-aH')|_{E} \otimes O_{X'}(bE)|_{E}\\
                   &\cong p_{1}^{*}O_{\mathbb{P}^{e-1}}(n-e+1)^{r_1}\otimes p_{2}^{*}M \otimes O_{X'}(-a\pi^{*}H)|_{E} \otimes O_{X'}((a+b)E)|_{E}\\
                   &\cong p_{1}^{*}O_{\mathbb{P}^{e-1}}(n-e+1)^{r_1}\otimes p_{2}^{*}M \otimes \pi^{*}O_{X}(-aH)|_{E} \otimes O_{E}(-a-b)\\
                   \end{align*}
Similarly as above consider the exceptional divisor $E \cong Z \times \mathbb{P}^{e-1}$ and two projections $p_{1}$ and $p_{2}$. So using the same identification as case 1 we have ,\\
\begin{equation*}
    F(-aH'+bE)|_{E} \cong p_{1}^{*}O_{\mathbb{P}^{e-1}}(n-e+1-a-b)^{r_{1}} \otimes p_{2}^{*}M(-aH)
\end{equation*}

Now we have to show that,$H^{i}(Z \times \mathbb{P}^{e-1}, p_{1}^{*}O_{\mathbb{P}^{e-1}}(n-e+1-a-b)^{r} \otimes p_{2}^{*}M(-aH))=0$ for all $i \geq 0$, $1 \leq a\leq n-e$ and all $ 1\le b \le n-e+1-a$. Using Kunneth formula we have,\\
\begin{equation*}
    H^{k}(Z \times \mathbb{P}^{e-1}, p_{1}^{*}O_{\mathbb{P}^{e-1}}(1-a-b) \otimes p_{2}^{*}O_{Z}(-aH)) \cong \bigoplus_{i+j=k}H^{i}(\mathbb{P}^{e-1},O_{\mathbb{P}^{e-1}}(1-a-b)) \otimes H^{j}(Z,O_{Z}(-aH) 
\end{equation*}
In this case, the vanishing on the projective space factor is not sufficient. Instead, we use $M$ is Ulrich bundle on $(Z, H|_{Z})$. Namely,\\
\begin{equation*}
    H^{i}(Z, M(-aH))=0 \ for \ all  \ i\ge 0 \ and \ 1\le a \le n-e 
\end{equation*}
 Thus all the terms using Kunneth formula vanishes. So Combining the results we get that,$H^{k}(E,F(-aH'+bE)|_{E})=0$ for $k\geq 0$ and $1 \leq a\leq n-e$ and $1 \le b \le n-e+1-a$\\
Now let's consider the long exact sequence of cohomology by,\\
\begin{align*}
    0 \rightarrow \Gamma(X',F(-aH'+(b-1)E) \rightarrow \Gamma(X',F(-aH'+bE)) \rightarrow &\Gamma(X'. F(-aH'+bE)|_{E})\\                   &\rightarrow H^{1}(F(-aH'+bE) \rightarrow ...               
\end{align*}    
Now as $F$ is Ulrich on $(X',H')$, so we have $H^{i}(X',F(-aH')=0$ for $1\leq a\leq n-e$ and for $b=0$. Thus we can start induction from $b \ge 1$. We see that for $b=1$, we have $H^{i}(F(-aH'+bE)=0$ for $1\leq a \leq n$. Let's assume that $H^{i}(F(-aH'+kE)=0$ for some $b=k$. Using the above long exact sequence of cohomology and the inductive hypothesis we have that $H^{i}(F(-aH'+(k+1)E)=0$ for $1 \leq a \leq n-e$ and $ 1\le b \le n-e+1-a$. Combining the common range for $a$ and $b$ we have that,\\
\begin{equation*}
    H^{i}(F(-pH'-(p-n+e-1)E)=0 \ \forall i \geq 0 \ and \ 1 \leq p\leq n-e.
\end{equation*}
\end{proof}

\textbf{proof of theorem 1.2:}
Let $F$ be an Ulrich on $(X',H')$. From lemma \ref{lemma 8.2}, we have that $R^{j}\pi_{*}F((n-e+1)E)=0$ for all $j>0$. Using \cite[page 252,Ex 8.1]{Hart} and the projection formula we obtain,  
\begin{equation*}
    H^{i}(\pi_{*}(F(n-e+1)E) \otimes O_{X}(-pH))) \cong H^{i}(F(-pH'-(p-n+e-1)E))
\end{equation*}
Using lemma \ref{lemma 8.3}, we have that $H^{i}(F(-pH'-(p-n+e-1)E)=0$ for all $i \geq 0$ and $1 \leq p \leq n$. Thus using the above isomorphism we conclude $H^{i}(\pi_{*}(F(n-e+1)E)=0$ for all $i \geq 0 $ and $1 \leq p \leq n$. So $\pi_{*}(F(n-e+1)E)$ is an Ulrich sheaf on $(X,H)$.\\

Now we show that if $F$ is an Ulrich bundle of rank $r$ on $(X',H')$, then $\pi_{*}F((n-e+1)E)$ is also an Ulrich bundle of rank $r$ on $(X,H)$.
We divide the proof in two cases,\\

\textbf{case 1: \ $z \in X\backslash Z $}\\
Note that the blow-up morphism is an isomorphism outside the variety $Z$. As it is an isomorphism so $\pi_{*}$ remains an exact functor. So we have that $rk(\pi_{*}(F(n-e+1)E)=r$.\\

\textbf{case 2:\ $z \in Z$}\\
 Clearly  for each $z \in Z$, the fiber of the morphism is $\mathbb{P}^{e-1}$ which is of constant dimension $e-1$. To prove that  the push forward is  locally free , it sufficies to verify that the dimension of $H^{0}$  that on each fiber is constant and  higher cohomology vanishes.
Thus following similar computation as above from lemma \ref{lemma 8.2} we  have,\\
\begin{equation*}
    F((n-e+1)E)|_{\mathbb{P}^{e-1}} \cong O_{\mathbb{P}^{e-1}}^{r}
\end{equation*}
 So computing the cohomology we see that,\\
 \begin{equation*}
     H^{0}(\mathbb{P}^{e-1}, O_{\mathbb{P}^{e-1}}^{r})=r
 \end{equation*}
 As all the higher cohomology vanishes so combining the two case we have $(\pi_{*}F((n-e+1)E)$ is also rank $r$ Ulrich bundle on $(X,H)$.
\begin{remark}
Note that, the  choice of a smooth projective variety $X$ and a subvariety $Z$ strictly determines the  normal bundle $N_{Z/X}$. In our construction the fundamental reason for choosing $X$ and $Z$ both abelian makes the normal bundle trivial which simplifies the cohomological computation significantly. But there are several choices for $X$ and $Z$ possible without being any one of them  to be abelian but such  that still the normal bundle is trivial. Our construction works verbatim in that set up also.
\end{remark}

\textbf{Acknowledgments}

We are extremely grateful to Prof. Angelo Felice Lopez for several suggestions and having several useful discussions during the period of work. We are thankful for cohesive research ambience of IISER Berhampur during this work. The first author acknowledges the National Board of Higher Mathematics(NBHM), Government of India,for the financial support with NBHM Fellowship No. 0203/13(35)/2021-RD-II/13162.


\begin{thebibliography}{00}
\bibitem[1]{Goto}
S.Goto,K.Ozeki,R.Takahashi,K.-I Waatneb,K.-I. Yoshida,Ulrich ideals and modules, Math.Proc.Cambridge Philos Soc,156(2014),137-166\\

\bibitem[2]{Arnuad}
Beauville, Arnaud, An introduction to Ulrich bundles, Eur.J.Math,2018,pp 26--36\\

\bibitem[3]{Hart}
Hartshorne, Robin, Algebraic geometry, Graduate Texts in Mathematics,  VOLUME No. 52, Springer-Verlag,New York-Heidelberg,1977\\

\bibitem[4]{Costa}
Costa, Laura and Mir o-Roig, Rosa Maria and Joan pons-Llopis, Ulrich bundles---from commutative algebra to algebraic geometry, De Gruyter in Mathematics, Volume 77,2021\\

\bibitem[5]{Jan}
Cheah  Jan,  Cellular decompositions for nested {H}ilbert schemes of
points, Pacific J. Math.,1998, pp 39--90\\

\bibitem[6]{Fog}
Fogarty John,Line bundles on quasi symmetric power of varieties,Journal of Algebra, pages 169-180,1977\\

\bibitem[7]{Secci}
Secci,Saverio Andrea, On the existence of Ulrich bundles on blown-up varieties at a point,Boll. Unione Mat.Ital pp-131-135,2020\\

\bibitem[8]{Kim}
Kim, Yeongrak, Ulrich bundles on blowing ups, C.R.Math.Acad.Sci.Paris,pp 1215-1218, 2016\\ 

\bibitem[9]{Eisenbud}
Eisenbud, David and Schreyer, Frank-Olaf,Resultants and {C}how forms via exterior syzygies,J. Amer. Math. Soc,pp-537-579, 2003\\

\bibitem[10]{Ein} 
Bertram, Aaron and Ein, Lawrence and Lazarsfeld, Robert,Vanishing theorems, a theorem of {S}everi, and the equations defining projective varieties,J. Amer. Math. Soc,pp 587-602,1991\\

\bibitem[11]{Casalenas}
M.Casanellas and R.Hartshrone,stable Ulrich bundles,Math. Ann. 370(2018),no 3-4, 1155-1180\\

\bibitem[12]{Casnati}
Casnati, Gianfranco,Kim, Yeongrak, Ulrich bundles on blowing up (and an erratum) [{MR}3573932], C. R. Math. Acad. Sci. Paris,pp 1291--1297, 2017\\



\bibitem[13]{BS}
M.C.Beltrametti and A.J. Sommese, Notes on Embeddings of Blow ups, Journal of algebra 186,pp 861-871(1996)\\

\bibitem[14]{Casnati}
Gianfranco Casnati, Special Ulrich bundles on non-special surface with $p_{g}=q=0$, International Journal of Mathematics,Vol 28, 18 pages,2017\\

\bibitem[15]{Anghel}
Cistian Anghel, Smooth Counterexamples to the Eisenbud--Schreyer--Weyman Ulrich Existence Problem, https://arxiv.org/abs/2609.02718\\

\bibitem[16]{Lopez}
Angelo Felice Lopez, On the existence of smooth varieties of any dimension carrying no Ulrich bundles for some very ample line bundle, https://arxiv.org/abs/2609.06678\\




\end{thebibliography}
\end{document}